\documentclass[12pt]{amsart}
\usepackage[utf8]{inputenc}
\usepackage{enumerate,enumitem,amsmath,amssymb,mathrsfs,stmaryrd}
\usepackage{amsthm}
\usepackage{latexsym,epsf,graphicx,comment,appendix}

\usepackage{hyperref} 
\usepackage{xcolor}

\hypersetup{
        colorlinks=true, 
        linktoc=all,     
        linkcolor=blue,  
        citecolor=orange, 
        filecolor=magenta, 
        urlcolor=black    
    }

\newfont{\bb}{msbm10 at 12pt}
\newfont{\tbb}{msbm10 at 8pt}

\def\R{\r}
\def\H{\mathbb{H}}

\def\R{\mathbb{R}}
\def\G{\mathbb{G}}
\def\L{\mathcal{L}}
\def\C{\mathcal{C}}

\def\bC{\mathbb{C}}

\baselineskip 
\usepackage {amsmath}
\usepackage {amsthm}
\usepackage{times}
\usepackage{amscd}
\usepackage{epsf}

\numberwithin{equation} {section}

\begin{document}
\mbox{}\vspace{0.2cm}\mbox{}

\providecommand{\keywords}[1]
{
  \small	
  \textbf{\textit{Keywords---}} #1
}

\theoremstyle{plain}\newtheorem{theorem}{Theorem}[section]
\theoremstyle{plain}\newtheorem{lemma}[theorem]{Lemma}
\theoremstyle{plain}\newtheorem{proposition}[theorem]{Proposition}
\theoremstyle{plain}\newtheorem{definition}[theorem]{Definition}
\theoremstyle{plain}\newtheorem{corollary}[theorem]{Corollary}
\theoremstyle{plain}\newtheorem{remark}[theorem]{Remark}

\theoremstyle{plain}\newtheorem*{theorem*}{Theorem}
\theoremstyle{plain}\newtheorem*{main theorem}{Main Theorem} 
\theoremstyle{plain}\newtheorem*{lemma*}{Lemma}
\theoremstyle{plain}\newtheorem*{claim}{Claim}
\theoremstyle{plain}\newtheorem*{corollary*}{Corollary}

\theoremstyle{plain}\newtheorem*{theorem 3.4}{Theorem 3.4}
\theoremstyle{plain}\newtheorem*{theorem 2.3}{Theorem 2.10}
\theoremstyle{plain}\newtheorem*{corollary 3.1}{Theorem 3.4}
\theoremstyle{plain}\newtheorem*{theorem 2.12}{Theorem 2.11}
\theoremstyle{plain}\newtheorem*{theorem 4.1}{Theorem 4.2}
\theoremstyle{plain}\newtheorem*{corollary 4.5}{Corollary 4.8}

\theoremstyle{definition}\newtheorem{example}{Example}[section]
\theoremstyle{definition}\newtheorem{acknowledge}{Acknowledgment}
\theoremstyle{definition}\newtheorem{conjecture}{Conjecture}

\begin{center}
\rule{15cm}{1.5pt} \vspace{.4cm}

{\bf\Large On potentials for sub-Laplacians and geometric applications} 
\vskip .3cm

Shiguang Ma$\mbox{}^\dag$ and Jie Qing$\mbox{}^\ddag$
\vspace{0.3cm} 
\rule{15cm}{1.5pt}
\end{center}


\title{}

\begin{abstract} In this paper we extend the research on potential theory and its geometric applications 
from Euclidean spaces to homogeneous Carnot groups. We introduce a new approach to use the 
geometric completeness to estimate the Hausdorff dimension of polar sets of potentials of nonnegative Radon 
measures for sub-Laplacians in homogeneous Carnot groups. Our 
approach relies on inequalities that are analogous to the classic integral inequalities about Riesz potentials in Euclidean spaces. Our approach 
also uses extensions of some of geometric measure theory to homogeneous Carnot groups and the polar coordinates 
with horizontal radial curves constructed by Balogh and Tyson \cite{BT} for polarizable Carnot groups.  
As consequences, we develop applications of potentials for sub-Laplacians in CR geometry, 
quaternionic CR geometry, and octonionic CR geometry.
\end{abstract}


\maketitle



\section{Introduction}

In this paper we want to extend our study of potential theory and geometric applications 
beyond conformal geometry. There have been extensive studies
of potential theory in a broad class of spaces. On that, our major references for this paper are \cite{BLU} and \cite{BB}. 
We want to focus on potentials of Radon
measures for sub-Laplacians in homogeneous Carnot groups and applications in CR geometry, quaternionic CR geometry, and octonionic CR geometry in this paper.

As in \cite{BLU}, the potentials of Radon measures $\mu$ with the support in a bounded domain $D\subset \R^N$ for the canonical sub-Laplacian $\Delta_\G$ 
in a homogeneous Carnot group $\G = (\R^N, \circ, \delta_\lambda)$ is 
$$
R^{\mu, D}_{\Delta_\G} (x) = \int_D \Gamma_{\Delta_\G} (x^{-1}\circ y) d\mu(y),
$$
where $\Gamma_{\Delta_\G}$ is the fundamental solution for the canonical sub-Laplacian $\Delta_\G$ in $\G$. 

Instead of relying on the property of thin sets in potential theory observed in \cite{AH, Mi, MQ-g, MQ-2025, MQ-s}, 
we develop a completely new way to utilize geometric completeness to
derive estimates on the size of polar sets of potentials of nonnegative Radon measures. 
Our approach uses new integral inequalities about potentials for sub-Laplacians in homogeneous Carnot groups 
that are analogous to classic integral inequalities 
\begin{equation}\label{Equ:intr-riesz}
\int_\Omega \frac 1{|x-y|^{N-a} |y-p|^{N-b}} dy \leq  \frac C{|x-p|^{N-(a+b)}} \quad  \text{ when $a+b < N$},
\end{equation}
for a bounded domain $\Omega$ in Euclidean space $\R^N$ and $a, b\in (0, N)$ (cf. \cite[Proposition 4.2]{Au} and \cite{Gi}). Our analog
of \eqref{Equ:intr-riesz} in homogeneous Carnot groups is

\begin{theorem 2.12} Let $\G = (\R^N, \circ, \delta_\lambda)$ be a homogeneous Carnot group and $d_{\Delta_\G}$ be
the $\Delta_\G$-gauge in $\G$ in \eqref{Equ:L-gauge}. Then, given a bounded domain $\Omega\subset R^N$, for a constant $C>0$, 
\begin{equation}\label{Equ:intr-sub-Lap}
\int_\Omega \frac 1{d_{\Delta_\G}(x^{-1}\circ y)^{Q-a} d_{\Delta_\G}(p^{-1}\circ x)^{Q-b}} dx \leq  
\frac C{d_{\Delta_\G}(p^{-1}\circ y)^{Q-(a+b)}}
\end{equation}
provided $a, b\in (0, Q)$ and $a+b < Q$, where $Q$ is the homogeneous dimension of $\G$.
\end{theorem 2.12}

Based on a geometric measure theoretic argument and \eqref{Equ:intr-riesz}, 
we demonstrate our new idea in the alternative proof of the estimates on
the size of polar sets of Newtonian potentials of nonnegative Radon measures 
(cf. Theorem \ref{Thm:new-SY-potential} in Section \ref{Sect:potential-result}). In order to implement our new idea 
for homogeneous Carnot groups, we recall the polar coordinates with horizontal radial curves constructed in \cite{BT} and
\eqref{Equ:polar-integral} on polarizable Carnot groups (cf. \cite[Definition 2.12]{BT}).  We also extend
some of geometric measure theory in the context of homogeneous Carnot groups.  We then are able to obtain our main result on potentials. 

\begin{theorem 3.4}\label{Thm:potential-carnot-intr} Let $\G=(\R^N, \circ, \delta_\lambda)$ be a polarizable Carnot group (cf. \cite[Definition 2.12]{BT}).
And let $R^{\mu, D}_{\Delta_\G}$ be the potential of a nonnegative Radon measure $\mu$ supported in $D\subset \R^N$ for the sub-Laplacian 
$\Delta_\G$ in $\G$.  Suppose that  $S$ is a compact subset in $D$ and that
\begin{equation}\label{Equ:complete-carnot-intr}
\lim_{t\to 0} R^{\mu, D}_{\Delta_\G}(\gamma(t)) = \infty \text{ and } \int_0^\delta (R^{\mu, D}_{\Delta_\G} (\gamma(t)))^\frac 2{Q-2} dt = \infty,
\end{equation}
for any curve $\gamma(t): [0, \delta]\to \R^N$ with $\gamma(0)\in S$ and $\gamma\in C[0, \delta]\bigcap C^1(0, \delta]$
that is horizontal with respect to the sub-Riemannian structure of $\G$ and
some $\delta >0$, where $Q$ is the homogeneous dimension of $\G$. Then 
\begin{equation}\label{Equ:intr-dim-p}
dim_{\mathcal{H}_{\Delta_\G}} (S) \leq \frac {Q-2}2.
\end{equation}
\end{theorem 3.4}

The Hausdorff measures $\mathcal{H}^s_{\Delta_\G}$ and the Hausdorff dimensions $dim_{\mathcal{H}_{\Delta_\G}}$ 
are with respect to the $\Delta_\G$-gauge $d_{\Delta_\G}$ in $\G$ in \eqref{Equ:L-gauge} discussed in \cite[Chapter 13]{BLU}. Clearly Theorem
\ref{Thm:potential-carnot} applies to $\Delta_\G$-superharmonic functions in the light of Theorem \ref{Thm:Riesz-l-superh} (cf. 
\cite[Theorem 9.4.4]{BLU}). It is important for us  that H-type groups (cf. \cite[Definition 5.1]{BT}  and \cite{Kap}) 
are polarizable by \cite[Proposition 5.6]{BT}, and that
H-type groups include Heisenberg groups, quaternionic Heisenberg groups, and the octonionic Heisenberg group, that are the nilpotent part of the
Iwasawa decomposition of rank one groups $SU(n+1, 1)$, $Sp(n+1, 1)$ and $\mathbb{F}_{4(-20)}$ (cf. \cite{Kap, BT, CDKR-91}). 

For convenience, from now on,
we always use $\mathsf{N}$ to stand for a H-type group that is the nilpotent part of the Iwasawa decomposition of a rank one group $\mathfrak{G}$,
where $\mathfrak{G}$ stands for a rank one group among $SU(n+1, 1)$, $Sp(n+1, 1)$ and $\mathbb{F}_{4(-20)}$. As applications of Theorem \ref{Thm:potential-carnot} to $\Delta_\G$-superharmonic functions in CR geometry, quaternionic CR geometry, 
and octonionic CR geometry in the light of \eqref{Equ:CR-yamabe}, \eqref{Equ:QC-yamabe}, and \eqref{Equ:OC-yamabe} respectively, we have

\begin{theorem 4.1} Assume $\mathsf{N}$ 
is a H-type group that is the nilpotent part of the Iwasawa decomposition of a rank one group $\mathfrak{G}$ among
$SU(n+1, 1)$, $Sp(n+1, 1)$ and $\mathbb{F}_{4(-20)}$. 
Let $S$ be a compact subset of a bounded and open subset $\Omega$ in $\mathsf{N}$. 
And let $\theta = u^\frac 4{Q-2}\theta_0$ be a 
conformal contact form on $\Omega\setminus S$, where $\theta_0$ is the standard contact form on $\mathsf{N}$ and $Q$ is the homogeneous 
dimension of $\mathsf{N}$. 
Suppose that the sub-Riemannian metric associated with $\theta$ is 
geodesically complete near $S$. Then
\begin{equation}\label{Equ:intr-dim-CR}
dim_{\mathcal{H}_{\Delta_{\mathsf{N}}}} (S) \leq \frac {Q-2}2
\end{equation}
when the scalar curvature $R_\theta$ of the conformal contact form $\theta$ is nonnegative.
\end{theorem 4.1}

The scalar curvature of a contact structure with respect to a H-type group $\mathsf{N}$ in Theorem \ref{Thm:app-CR}
is recalled in Section \ref{Subsect:H-type} (cf. \cite{JL, Bq, WW, IMV, SW}). 
Theorem \ref{Thm:app-CR} is an extension of our previous work \cite[Theorem 1.3]{MQ-2025} in conformal geometry. In fact
\eqref{Equ:intr-dim-CR} still holds when the scalar curvature $R_\theta$ is not necessarily nonnegative 
everywhere (please see Theorem \ref{Thm:app-CR-2.0} in Section \ref{Sect:geo-app}). Theorem \ref{Thm:app-CR} also holds for domains in manifolds with the corresponding structure,
since our approach is all local in nature. Theorem \ref{Thm:app-CR-2.0} for CR case with respect to Heisenberg groups 
has been announced in \cite{CY}.
The number $\frac {Q-2}2$ in \eqref{Equ:intr-dim-CR} is sharp evidently by \cite{GMM}.

Motivated by \cite{Nay-c}, in \cite{Nay-h, SW-n, WW}, it was shown that the scalar curvature of the Carnot-Carath\'{e}odory metrics from Nayatani's
construction on $\Omega(\Gamma)/\Gamma$ for a convex cocompact subgroup $\Gamma$ of a rank one group $\mathfrak{G}$
is positive if and only if  $dim_\mathcal{H} (L(\Gamma)) < \frac {Q-2}2$. Note that
Patterson-Sullivan Theorem has been generalized to geometrically finite subgroups of rank one groups $\mathfrak{G}$
(cf. \cite{Cor, CI, DK}) and therefore the Poincar\'{e} critical exponent and the Hausdorff dimension of the limit point
set are the same here. We also note that Theorem \ref{Thm:app-CR}, which relies on potential theory and local analysis, 
implies the following corollary that extends to all 
rank one groups $\mathfrak{G}$ from \cite[Proposition 2.4]{SY} in conformal groups.

\begin{corollary 4.5} Let $\mathsf{N}$  be the H-type group that is the nilpotent part of the Iwasawa 
decomposition of a rank one group $\mathfrak{G}$ as in Theorem \ref{Thm:app-CR}. And let 
$\Gamma$ be a convex cocompact subgroup of $\mathfrak{G}$. 
Suppose that the Yamabe type constant $\lambda(\Omega(\Gamma)/\Gamma) \geq 0$ in \eqref{Equ:yamabe-type}. Then
\begin{equation}
dim_{\mathcal{H}}(L(\Gamma)) \leq \frac {Q-2}2,
\end{equation}
where $Q$ is the homogeneous dimension of the group $\mathsf{N}$.
\end{corollary 4.5}

The organization of this paper is as follows: In Section \ref{Sect:homo-carnot}, we introduce potentials of Radon
measures for sub-Laplacians in homogeneous Carnot groups 
based on \cite{BLU}. Then we prove our new integral inequalities \eqref{Equ:intr-sub-Lap}.
In Section \ref{Sect:potential-result}, we first use our new idea and \eqref{Equ:intr-riesz} to present the alternative proof of 
the estimates of the size of polar sets of Newtonian potentials of nonnegative Radon measures in Euclidean spaces. Then we review 
the polar coordinates in \cite{BT} and develop tools in geometric measure theory before proving our main theorem on potentials 
with the new idea and \eqref{Equ:intr-sub-Lap}. 
Finally in Section \ref{Sect:geo-app}, we apply Theorem \ref{Thm:potential-carnot} to prove our geometric applications on the size of
singularities of the conformal contact structure of nonnegative scalar curvature in relevant CR geometry associated with
the rank one groups $\mathfrak{G}$.  


\section{On potentials for sub-Laplacians in homogeneous Carnot groups} \label{Sect:homo-carnot}

In this section we follow the book \cite{BLU} to introduce potentials of Radon measures 
for sub-Laplacians in homogeneous Carnot groups. Our main goal in this section is to present our new integral inequalities
about potentials for sub-Laplacians in homogeneous Carnot groups. 


\subsection{Basic notations and definitions} \label{SubSect:definition} Let us start with the definition of homogeneous Carnot groups.

\begin{definition} \label{Def:hcg} (\cite[Definition 1.4.1]{BLU}) We say a Lie group on $\mathbb{R}^N$, $\G= (\mathbb{R}^N,   \circ, \delta_\lambda)$, 
is a homogeneous Carnot group if the following properties hold
\begin{enumerate}
\item $\mathbb{R}^N$ can be split as $\mathbb{R}^{N_1}\times\mathbb{R}^{N_2}\cdots\times\mathbb{R}^{N_r}$, and the dilation
$$
\delta_\lambda (x^{(1)}, x^{(2)}, \cdots x^{(r)}) = (\lambda x^{(1)}, \lambda^2 x^{(2)}, \cdots \lambda^r x^{(r)}), \text{ where $x^{(i)}\in \mathbb{R}^{N_i}$,}
$$
is an automorphism of the group $\mathbb{G}$ for all $\lambda > 0$. Hence the group with this dilation $\delta_\lambda$ 
is a homogeneous Lie group on $\R^N$ (cf. \cite[Definition 1.3.1]{BLU}). 
\item Let $Z_1, Z_2, \cdots, Z_{N_1}$ be the left invariant vector fields on $\mathbb{G}$ such that
$$
Z_i (0) = \frac {\partial}{\partial x_i}|_{0} \text{ for $i= 1, 2, \cdots N_1$}.
$$
Then the rank of the Lie algebra generated by $Z_1, Z_2, \cdots Z_{N_1}$ is $N$. 
\end{enumerate}
The group together with the dilation $(\mathbb{R}^N, \circ, \delta_\lambda)$ is said to be a homogeneous Carnot group of step $r$ and $N_1$ generators. And
the vector fields $Z_1, Z_2, \cdots Z_{N_1}$ are called the Jacobian generators.
\end{definition}

By \cite[Proposition 1.3.21]{BLU}, the standard Lebesgue measure $m$ 
on $\mathbb{R}^N$ is the Haar measure. And, for a measurable subset $E$,
\begin{equation}\label{Equ:Lebesgue-scaling}
m(\delta_\lambda (E))= \lambda^Qm(E),
\end{equation}
where $Q= \sum_{i=1}^r iN_i$
is called the homogeneous dimension of the homogeneous Carnot group $\mathbb{G}$ (cf. \cite[Proposition 2.2.8]{BLU}). 
For our purpose we assume $Q>2$ throughout this paper. 
Next important objects for us are sub-Laplacians in homogeneous Carnot groups.  

\begin{definition}\label{Def:sub-L} (\cite[Definition 1.5.1]{BLU})  Let $Z_1, Z_2, \cdots, Z_{N_1}$ be the Jacobian generators of a homogeneous Carnot group
$\mathbb{G} = (\mathbb{R}^N, \circ, \delta_\lambda)$ as in Definition \ref{Def:hcg}. The second order differential operator
$$
\Delta_{\mathbb{G}} = \sum_{i=1}^{N_1} Z_i^2
$$
is called the canonical sub-Laplacian (or the sub-Laplacian if no confusions) on $\mathbb{G}$. 
\end{definition}

Remarkably,  the rank condition in H\"{o}rmander Theorem (cf. \cite[Theorem 1 in Preface]{BLU}) in analysis
is exactly the assumption in \cite[Theorem 19.1.3]{BLU} by
Carath\'{e}redory, Chow, and Rashevsky, in the theory of sub-Riemannian geometry, that introduces the 
Carnot-Carath\'{e}redory distance $d_{\G}$ (or the control distance) for homogeneous Carnot groups.  
A homogeneous Carnot group $\mathbb{G}$ together with the control distance $d_\G$ 
is a metric space and called a Carnot-Carath\'{e}redory space according to \cite[A.6]{BB}. Next important objects for us are fundamental solutions and gauge functions. 

\begin{definition}\label{Def:fund-sol} (\cite[Definition 5.3.1]{BLU})
Let $\mathbb{G} = (\mathbb{R}^N, \circ, \delta_\lambda)$ be a homogeneous Carnot group. A function
$\Gamma_{\Delta_\G}: \mathbb{R}^N\setminus \{0\}\to \mathbb{R}$ is said to be a fundamental solution for $\Delta_\G$ if:
\begin{itemize}
\item $\Gamma_{\Delta_\G}$ is smooth away from the origin;
\item $\Gamma_{\Delta_\G} \in L^1_{\text{loc}}(\mathbb{R}^N)$ and $\Gamma_{\Delta_\G} (x)\to 0$ as $x\to \infty$;
\item $-\Delta_\G \Gamma_{\Delta_\G} = \text{Delta}_0,$ where $\text{Delta}_0$ is the Dirac measure supported at the origin $0$. More explicitly,
$
\int_{\mathbb{R}^N} \Gamma_{\Delta_\G} \Delta_\G \phi = -\phi(0), \forall \phi\in C^\infty_0(\mathbb{R}^N).
$
\end{itemize}
\end{definition} 

It turns out, by \cite[Theorem 5.3.2 and Proposition 5.3.10]{BLU}, there always exists a unique fundamental solution for $\Delta_\G$, 
which is positive with the homogeneity 
\begin{equation}\label{Equ:Gamma-scale}
\Gamma_{\Delta_\G} (\delta_\lambda (x)) = \lambda^{2 - Q}\Gamma_{\Delta_\G} (x)
\end{equation}
(cf. \cite[Proposition 5.3.12 and Proposition 5.3.13]{BLU}),
where $Q$ is the homogeneous dimension of $\mathbb{G}$. 
The $\Delta_\G$-gauge
\begin{equation}\label{Equ:L-gauge}
d_{\Delta_\G} (x, y)= d_{\Delta_\G} (y^{-1}\circ x) = \left\{\aligned \Gamma_{\Delta_\G} (y^{-1}\circ x)^{\frac 1{2-Q}} & \text{ when $x\neq y$}\\
0 \quad & \text{ when $x=y$}\endaligned\right.
\end{equation}
is a symmetric homogeneous norm on $\mathbb{G}$ (cf. \cite[Definition 5.4.1]{BLU}). It is good to be reminded that all homogeneous norms 
including $d_\G$ and $d_{\Delta_\G}$ in $\G$ are equivalent according to \cite[Proposition 5.1.4]{BLU}.

\subsection{$\Delta_\G$-superharmonic functions and potentials of nonnegative Radon measures for sub-Laplacians}\label{SubSec:superharmonic} 
Let us first recall discussions in \cite[Section 7.2]{BLU} and the following definition of $\mathcal{L}$-superharmonic functions.

\begin{definition}\label{Def:l-superharm} (\cite[Definition 7.2.2 and Remark 7.2.3]{BLU}) Let $\mathbb{G}= (\mathbb{R}^N, \circ, \delta_\lambda)$ be a homogeneous Carnot group.
Let $\Omega$ be an open subset of $\mathbb{R}^N$. A function 
$$
u: \Omega \to (-\infty, \infty]
$$
is said to be $\Delta_\G$-superharmonic in $\Omega$ if 
\begin{itemize}
\item $u$ is lower semi-continuous and $\{x\in \Omega: u(x) <\infty\}$ is a dense subset of $\Omega$;
\item for every $\Delta_\G$-regular open set $V$ (cf. \cite[Section 7.2]{BLU}) with the closure $\bar V\subset \Omega$ and for every $x\in V$,
$$
H^V_\phi (x) \leq u(x)
$$
for every continuous  function $\phi\in C(\partial V, \mathbb{R})$ such that $\phi \leq u$ on $\partial V$, where $H^V_\phi$ is the unique solution to
$$
- \Delta_\G H^V_\phi = 0 \text{ in $V$ and } H^V_\phi = \phi \text{ on $\partial V$},
$$
and the existence and uniqueness of $H^V_\phi$ is because the open set $V$ is assumed to be  $\Delta_\G$-regular (cf. \cite[Section 7.2]{BLU}). 
\end{itemize}
\end{definition}

Next let us introduce the Riesz representation of $\Delta_\G$-superharmonic functions and potentials for the sub-Laplacian $\Delta_\G$ in homogeneous
Carnot groups.

\begin{theorem} \label{Thm:Riesz-l-superh} (\cite[Theorem 9.4.4]{BLU})  Let $\mathbb{G}= (\mathbb{R}^N, \circ, \delta_\lambda)$ be a 
homogeneous Carnot group. Suppose $\Omega\subset \R^N$ and $u$ is a $\Delta_\G$-superharmonic function on $\Omega$. Then, there is a nonnegative Radon
measure $\mu$ in $\Omega$, which is called the $\Delta_\G$-Riesz measure for $u$ (cf. \cite[Definition 9.4.1]{BLU}), and for all bounded open subset 
$\Omega'$ whose closure $\overline{\Omega'} \subset \Omega$ 
\begin{equation}\label{Equ:Riesz-rep}
u(x) = \int_{\Omega'} \Gamma_{\Delta_\G} (y^{-1}\circ x)d\mu(y) + h(x)
\end{equation}
for some  function $h$ which is $\Delta_\G$-harmonic in $\Omega'$, where $\Gamma_{\Delta_\G}$ is the fundamental solution for $\Delta_\G$.
\end{theorem}

At this point we are ready to introduce potentials of Radon measures for sub-Laplacians in homogeneous Carnot groups.

\begin{definition}\label{Def:potential} Let $\mathbb{G}= (\mathbb{R}^N, \circ, \delta_\lambda)$ be a 
homogeneous Carnot group. We define the potential of a Radon measure 
for the sub-Laplacian $\Delta_\G$ as
\begin{equation}\label{Equ:potential}
R_\G^{\mu, D} (x) = \int_D \Gamma_{\Delta_\G} (y^{-1}\circ x)d\mu(y)
\end{equation}
for all $x\in D\subset \R^N$, where $\mu$ is a Radon measure on $\R^N$ and the support of $\mu$ is in $D$. 
\end{definition}


\subsection{New integral inequalities about potentials for sub-Laplacians}\label{Subsect:integral}
In recent works \cite{MQ-g, MQ-2025, MQ-s} on applications of potential theory in conformal geometry, the idea, started in \cite{Hu, AH}, 
relies vitally on the feature of thin sets in potential theories in Euclidean spaces, which is, for example, \cite[Theorem 1.1]{AH} , 
\cite[Theorem 8.1 in Chapter 2]{Mi}, and \cite[Theorem 1.2]{MQ-2025}.  
In this paper, we are taking a completely different approach. 
It turns out one of the key analytic ingredients of our new approach in geometric applications of potential theory 
is the integral inequalities \eqref{Equ:riesz-like}, which are extensions of the classic integral inequalities about Riesz potential in Euclidean spaces.  

\begin{proposition} \label{Prop:riesz-integral} (\cite[Proposition 4.12]{Au})
Let $\Omega\subset \R^N$ be a bounded domain. Then, for a constant $C>0$, 
\begin{equation}\label{Equ:ine-riesz}
\int_\Omega \frac 1{|x-y|^{N-a} |x-p|^{N-b}} dx \leq  \frac C{|y-p|^{N-(a+b)}}
\end{equation}
provided $a, b\in (0, N)$ and $a+b < N$.
\end{proposition}

\eqref{Equ:ine-riesz} was first proved by Giraud in \cite{Gi}.
Our analog of Proposition \ref{Prop:riesz-integral} in homogeneous Carnot groups is

 \begin{proposition} \label{Prop:new-integral} Let $\G = (\R^N, \circ, \delta_\lambda)$ be a homogeneous Carnot group
and $|\cdot|_\G$ be the symmetric homogeneous norm corresponding to the Carnot-Carath\'{e}redory 
distance $d_\G$ on $\G$.  Then, given a bounded domain $\Omega\subset R^N$, for a constant $C>0$, 
\begin{equation}\label{Equ:riesz-like}
\int_\Omega \frac 1{|x^{-1}\circ y|_\G^{Q-a} |p^{-1}\circ x|_\G^{Q-b}} dx \leq  \frac C{|p^{-1}\circ y|_\G^{Q-(a+b)}},
\end{equation}
provided $a, b\in (0, Q)$ and $a+b < Q$, where $Q$ is the homogeneous dimension of $\G$.
\end{proposition}

The proof of Proposition \ref{Prop:new-integral} turns out to straightforward in the same spirit of that of Proposition \ref{Prop:riesz-integral}. 
Let us first prepare some basic facts about the Lebesgue measure on $\G$. 

\begin{lemma} (\cite[Proposition 1.3.21]{BLU}) \label{Lem:lebesgue-1}
The standard Lebesgue measure $m$ on a homogeneous Carnot group $\G = (\R^N, \circ, \delta_\lambda)$ is the Haar measure on $\G$ 
and homogeneous with respect to the dilation, that is, for a measurable subset $E$ and $\lambda>0$,
$m(\delta_\lambda (E)) = \lambda^Q m(E)$ and,  in particular, 
\begin{equation}\label{Equ:ball-v}
m(B_\G(x, r)) = c r^Q
\end{equation}
for a ball $B_\G(x, r)$ with respect to the control distance, where $c$ is a constant for the group 
$\G$.
\end{lemma}

\begin{lemma}\label{Lem:lebesgue-2} 
Let $\G$ be a homogeneous Carnot group $\G = (\R^N, \circ, \delta_\lambda)$. Then
$$
\int_{B_\G(p, r)} \frac 1{|p^{-1}\circ x|_\G^s} dx \leq c r^{Q-s}
$$
for $0< s< Q$, where $c$ is a constant for the group $\G$. Similarly, 
$$
\int_{B^c_\G(p, r)} \frac 1{|p^{-1}\circ x|_\G^s} dx \leq c r^{Q-s}
$$
for $s>Q$, where $c$ is a constant for the group $\G$.
\end{lemma}

\begin{proof} Without loss of generality, let $p=0$. We write
$$
\int_{B_\G(0, r)} |x|_\G ^{-s}dx  = \sum_{k=1}^\infty \int_{B_\G(0, \frac r{2^{k-1}})\setminus B_\G(0, \frac r{2^k})}
|x|_\G^{-s} dx  \leq c (\sum_{k=1}^\infty (2^{s-Q})^k ) r^{Q-s} \leq c r^{Q-s}.
$$
The other one can be proved similarly.
\end{proof}

We are now ready to carry out the proof of Proposition \ref{Prop:new-integral} with the help
from Lemma \ref{Lem:lebesgue-1} and Lemma \ref{Lem:lebesgue-2}.

\begin {proof}[Proof of Proposition \ref{Prop:new-integral}] Without loss of generality, we may assume $y=0$.
We first decompose the integral into three parts
$$
\int_\Omega \frac 1{|x|_\G^{Q-a} |x-p|_\G^{Q-b}} dx = I +II +III,
$$
where
$$
I = \int_{\Omega\cap B_\G(p, \frac 12|p|_\G)} \frac 1{|x|_\G^{Q-a} |x-p|_\G^{Q-b}} dx, \ 
III = \int_{\Omega\setminus B_\G(p, 3|p|_\G)} \frac 1{|x|_\G^{Q-a} |x-p|_\G^{Q-b}} dx
$$
and
$$
II = \int_{\Omega\cap (B_\G(p, 3|p|_\G)\setminus B_\G(p, \frac 12|p|_\G))} \frac 1{|x|_\G^{Q-a} |x-p|_\G^{Q-b}} dx.
$$
For $I$, we observe
$
|x|_\G \geq |p|_\G-|x-p|_\G > \frac 12|p|_\G.
$
Hence, for $a\leq Q$ and $b>0$,
$$
I \leq \frac C{|p|_\G^{Q-a}} \int_{B_\G(p, \frac 12|p|_\G)}\frac 1{|x-p|_\G^{Q-b}}dx \leq \frac C{|p|_\G^{Q-a-b}},
$$
where we use Lemma \ref{Lem:lebesgue-2}. For $II$, we further decompose it 
$$
II= II_1+II_2,
$$
where 
$$
II_1 =  \int_{\Omega\cap (B_\G(p, 3|p|_\G)\setminus B_\G(p, \frac 12|p|_\G))\cap B_\G(0, \frac 12|p|_\G)}
\frac 1{|x|_\G^{Q-a} |x-p|_\G^{Q-b}} dx
$$
and
$$
II_2 =  \int_{\Omega\cap (B_\G(p, 3|p|_\G)\setminus (B_\G(p, \frac 12|p|_\G)\cup B_\G(0, \frac 12|p|_\G)))} 
\frac 1{|x|_\G^{Q-a} |x-p|_\G^{Q-b}} dx.
$$
For $II_1$, we observe
$
|x-p|_\G\geq |p|_\G- |x|_\G > \frac 12|p|_\G
$ 
for $x\in B_\G(0, \frac 12|p|_\G)$. Hence, for $b\leq Q$ and $a>0$,
$$
II_1 \leq \frac C{|p|_\G^{Q-b}} \int_{B_\G(0, \frac 12|p|_\G)} \frac 1{|x|_\G^{Q-a}}dx \leq \frac C{|p|^{Q-a-b}},
$$
where we use Lemma \ref{Lem:lebesgue-2} again. On the other hand, for $a, b \leq Q$,
$$
II_2 \leq \int_{B_\G(p, 3|p|_\G)\setminus (B_\G(p, \frac 12|p|_\G)\cup B_\G(0, \frac 12|p|_\G))} 
\frac 1{|x|_\G^{Q-a} |x-p|_\G^{Q-b}} dx\leq \frac C{|p|_\G^{2Q-a-b}}|p|_\G^Q = \frac C{|p|_\G^{Q-a-b}}
$$
due to \eqref{Equ:ball-v}. Finally, for $III$, we again observe
$
|x|_\G \geq |x-p|_\G -|p|_\G > 2|p|_\G.
$
Hence, for $a+b< Q$,
$$
III \leq \frac C{|p|_\G^{Q-a-b-\epsilon}} \int_{B^c_\G(p, 3|p|_\G)} \frac 1{|x-p|_\G^{Q+\epsilon}}dx \leq 
\frac C{|p|_\G^{Q-a-b}}
$$
for $\epsilon$ is small and positive, where we use the second inequality in Lemma \ref{Lem:lebesgue-2}.
\end{proof}

As a consequence of Proposition \ref{Prop:new-integral}, in the light of the 
equivalence between symmetric homogeneous norms $d_{\Delta_\G}$ and $d_\G$ (cf. \cite[Proposition. 5.1.4]{BLU}), we have
 
\begin{theorem}\label{Thm:new-integral-gauge} Let $\G = (\R^N, \circ, \delta_\lambda)$ be a homogeneous Carnot group. 
Then, given a bounded domain $\Omega\subset R^N$, for a constant $C>0$, 
\begin{equation}\label{Equ:new-inequality}
\int_\Omega \frac 1{d_{\Delta_\G}(x^{-1}\circ y)^{Q-a} d_{\Delta_\G}(p^{-1}\circ x)^{Q-b}} dx \leq  
\frac C{d_{\Delta_\G}(p^{-1}\circ y)^{Q-(a+b)}}
\end{equation}
provided $a, b\in (0, Q)$ and $a+b < Q$, where $Q$ is the homogeneous dimension of $\G$.
\end{theorem}


\section{On the size of polar sets of potentials for sub-Laplacians} \label{Sect:potential-result}

In this section we want to use \eqref{Equ:new-inequality} to show the estimate \eqref{Equ:intr-dim-p} on the size of polar sets of 
potentials $R^{\mu, D}_{\Delta_\G}$. To demonstrate, we want to first 
use our new approach and 
\eqref{Equ:ine-riesz} to give the alternative proof of the estimates on the Hausdorff dimension of polar sets of Newtonian potentials in
Euclidean spaces (see, for example, \cite[Theorem 1.3]{MQ-2025}). 
Then we extend our arguments and implement our new idea to derive \eqref{Equ:intr-dim-p} in homogeneous Carnot groups.

\subsection{On Newtonian potentials in Euclidean spaces and the new idea}\label{Subsect:euclidean}
Our new idea uses the integral inequalities \eqref{Equ:ine-riesz} 
about Riesz potentials in Euclidean spaces as well as the following facts from geometric measure theory in Euclidean spaces. 

\begin{lemma}\label{Lem:Falconer} (\cite[Theorem 5.4]{Fal}) Let $E$ be a closed subset of $\R^N$ with $\mathcal{H}^s(E)=\infty$. Then
\begin{itemize}
\item[(a)] Let $c$ be a positive number. Then there is a compact subset $F$ of $E$ such that $\mathcal{H}^s(F) = c$.
\item[(b)] There is a compact subset $F$ of $E$ such that $\mathcal{H}^s(F)>0$ and
\begin{equation}\label{Equ:Fal-2}
\mathcal{H}^s(F\cap B(x, r))\leq b r^s
\end{equation}
for $x\in \R^N, r\leq 1$ and some constant $b$.
\end{itemize}
\end{lemma}

Next is a simple and important fact for our argument. Let 
\begin{equation}\label{Equ:phi}
\phi (y) = \int_{F} \frac 1{|y-p|^t}d\mathcal{H}^{s} (p).
\end{equation}
Then $\phi(y)$ can be verified to be continuous with the help of Lemma \ref{Lem:Falconer}, provided that $0\leq t<s$.

\begin{lemma}\label{Lem:continuity-euclidean} 
Suppose that  $F\subset\R^N$ be a compact subset which satisfies 
$\mathcal{H}^{s}(F) \in (0, \infty)$ and \eqref{Equ:Fal-2}.
And assume that  $0\leq t< s$. Then $\phi(y)$ in \eqref{Equ:phi} is continuous with respect to $y\in \R^N$.
\end{lemma}

\begin{proof} We want to show,
$\forall \ y_0\in\R^n$ and $\forall \ \epsilon > 0$, $\exists \ \delta> 0$ such that 
$
|\phi(y) - \phi(y_0)| \leq \epsilon,
$
provided that $|y -y_0| < \delta$. We first observe that, for any $\bar\delta \leq 1$,
\begin{equation}\label{Equ:inside-small-ball-e}
\int_{F\cap B(y, \bar\delta)} \frac 1{|y -p|^t}d\mathcal{H}^{s}(p)
\leq C \bar\delta^{s - t}
\end{equation}
where a single $C$ can be chosen for all $y\in \overline{B(y_0, 2)}$ and $B(y_0, r)$ is the Euclidean ball with radius $r$ and centered at $y_0$. 
To prove \eqref{Equ:inside-small-ball-e}, we estimate it through dyadic annuli,  
\begin{equation}\label{Equ:003-e}
\aligned
& \int_{F\cap B(y, \bar\delta)}  \frac 1{|y - p|^t}d\mathcal{H}^{s}(p)  
 = \sum_{k=0}^\infty \int_{F\cap (B(y, \frac {\bar\delta}{2^k})\setminus B(y, \frac {\bar\delta}{2^{k+1}}))} 
\frac 1{|y - p|^t}d\mathcal{H}^{s}(p) \\
& \leq  \sum_{k=0}^\infty  (\frac {2^{k+1}}{\bar\delta})^t \mathcal{H}^{s} (F\cap B(y, \frac {\bar\delta} {2^k}))  
\leq  b \sum_{k=0}^\infty  (\frac {2^{k+1}}{\bar\delta})^t (\frac {\bar\delta} {2^k})^{s}  =  \frac {2^tb}{1 - 2^{t-s}} \bar\delta^{s-t}
\endaligned
\end{equation}
where we use \eqref{Equ:Fal-2}. We decompose the difference as follows:
$$
 \phi(y) - \phi(y_0) = \Phi_1+\Phi_2
= (\int_{F\cap B(y_0, \bar\delta)}  +
\int_{F\setminus B(y_0, \bar\delta)}) (\frac 1{|y-p|^t} -
\frac 1{|y_0-p|^t}) d\mathcal{H}^{s}(p).
$$
Let $\delta_1\in (0, 1]$ 
be such that $\frac {C2^t}{1 - 2^{t-s}}\delta_1^{s-t} < \frac \epsilon 4$ according to \eqref{Equ:003-e}. Then we get
$$
|\Phi_1| \leq \frac \epsilon 2 \text{ for $\bar\delta = \delta_2 =  \frac 1{2}\delta_1\leq \delta_1$ and $|y - y_0| < \delta_2$}.
$$
Because $B(y_0, \bar\delta) \subset B(y, \delta_1)$. Meanwhile
$$
|\Phi_2|  \leq \int_{F\setminus B(y_0, \bar\delta)}  \frac {||y_0-p|^t - |y - p|^t|}
{|y_0 - p|^t |y - p|^t}d\mathcal{H}^{s}(p)
\leq \frac C{\bar\delta^{2t}} \int_{F\setminus B(y_0, \bar\delta)}
||y- p|^t - |y_0- p|^t| d\mathcal{H}^{s}(p)
$$
for $|y -  y_0| < \delta_3 = \frac 1{2} \bar\delta$. Next, it is easily seen that there exists $\delta_4>0$ such that 
$$
 \int_{F} ||y - p|^t - |y_0- p|^t| d\mathcal{H}^{s}(p) \leq \frac {\bar\delta^{2t}}{2C} \epsilon
$$
if $|y - y_0|<  \delta_4$ by the reversed triangle inequality. So
$$
|\Phi_2| \leq \frac \epsilon 2 \text{ for $\bar\delta = \delta_2$ and $|y - y_0|< \min\{\delta_2, \delta_3, \delta_4\}$}.
$$
The proof of the lemma is completed.
\end{proof}

Now we are ready to present the alternative proof of the following result with our new idea. 

\begin{theorem}\label{Thm:new-SY-potential} Let $\mu$ be a nonnegative Radon measure with the support in $D\subset \R^N$ and
$$
R^{\mu, D} (x) = \int_D \frac 1{|x-y|^{N-2}}d\mu(y)
$$
be the Newtonian potential. Suppose that $S$ is a compact subset of  $D$ and that
\begin{equation}\label{Equ:complete-euclidean}
R^{\mu, D} (x) \to \infty \text{ as $x\to S$ and } \int_0^\delta (R^{\mu. D}(\gamma (t)))^\frac 2{N-2} dt = \infty
\end{equation}
for any $C^1$ curve $\gamma(t): [0, 1]\to \R^N$ with $\gamma(0) \in S$ and some $\delta>0$. Then, for $N\geq 4$,  
\begin{equation}\label{Equ:theorem 1.3}
dim_{\mathcal{H}} (S) \leq \frac {N-2}2.
\end{equation}
\end{theorem}

\begin{proof} Assume otherwise that $dim_{\mathcal{H}}(S) > \frac {N-2}2$ and consequently there is a compact subset $S'\subset S$ such that 
$$
0 < \mathcal{H}^{\frac {N-2}2 +\epsilon}(S') < \infty
$$
and \eqref{Equ:Fal-2} holds for $S'$, for some $\epsilon >0$ by Lemma \ref{Lem:Falconer}. Let us first assume $N\geq 5$.  
To contradict with \eqref{Equ:complete-euclidean}, we want to have 
$$
\int_0^\delta (R^{\mu, D}(p+r\sigma))^\frac 2{N-2} dr < \infty, \text{ for some $\delta>0$, some point $p\in S'$, and some direction 
$\sigma\in \mathbb{S}_1$.}
$$
For this, by the H\"{o}lder's inequality,
$$
\int_0^\delta (R^{\mu, D}(p+r\sigma))^\frac 2{N-2} dr \leq (\int_0^\delta R^{\mu, D}(p+r\sigma) r^\alpha dr)^\frac 2{N-2} (\int_0^\delta r^{-\frac {2\alpha}{N-4}}dr)^{1-\frac 2{N-2}},
$$
one realizes that it suffices to show $\int_0^\delta R^{\mu, D}(p+r\sigma) r^\alpha dr < \infty$
when $\alpha < \frac {N-4}2$. Next, for our purpose, we want to claim
\begin{equation}\label{Equ:Euclidean-integral}
\int_{S'}d\mathcal{H}^{\frac {N-2}2+\epsilon}(p) \, \int_{\mathbb{S}^{N-1}_1}d\sigma \int_0^\delta R^{\mu, D} (p+r\sigma) r^\alpha dr < \infty 
\text{ for some $\alpha < \frac {N-4}2$. }
\end{equation}
Hence, to achieve the claim \eqref{Equ:Euclidean-integral}, we apply Fubini's Theorem and rewrite the integral
\begin{equation}
\aligned
\int_{S'} & d\mathcal{H}^{\frac {N-2}2+\epsilon}(p) \, \int_{\mathbb{S}^{N-1}_1}d\sigma \int_0^\delta R^{\mu, D} (p+r\sigma) r^\alpha dr \\
= \int_{S'} & d\mathcal{H}^{\frac {N-2}2+\epsilon}(p) \, \int_{\mathbb{S}^{N-1}_1} d\sigma \int_0^\delta dr \int_D\frac 1{|p+r\sigma-y|^{N-2}}r^\alpha
d\mu(y) \\
 = \int_D & d\mu(y) \int_{S'}d\mathcal{H}^{\frac {N-2}2 + \epsilon}(p) \int_{B_\delta (p)}\frac {dx}{|x-y|^{N-2}|x-p|^{N-1-\alpha}}. \endaligned
\end{equation}
Then we apply Proposition \ref{Prop:riesz-integral} and get
$$
 \int_{B_\delta(p)} \frac 1{|x-y|^{N-2}|x-p|^{N-1-\alpha}}dx \leq \frac C{|y - p|^{N-3-\alpha}} = \frac C{|y-p|^{\frac {N-2}2+\frac \epsilon 2}}
$$ 
for $\alpha = \frac {N-4}2 - \frac \epsilon 2$. Therefore, by Lemma \ref{Lem:continuity-euclidean}, we know
$$
\int_{S'} \frac 1{|y-p|^{\frac {N-2}2+\frac \epsilon 2}}  d\mathcal{H}^{\frac {N-2}2+\epsilon} (p) 
$$
is continuous with respect to $y$ and  \eqref{Equ:Euclidean-integral} holds. Thus this completes the proof for $N\geq 5$. For $N=4$, the above argument works with $\alpha=0$ and no need to apply H\"{o}lder's  inequality to get
$$
\int_{S'}d\mathcal{H}^{1+\epsilon}(p) \, \int_{\mathbb{S}^{3}_1}d\sigma \int_0^\delta R^{\mu, D}(p+r\sigma) dr < \infty.
$$
Our new approach does not seem to work for $N=3$ for this theorem.
\end{proof}


\subsection{On potentials for sub-Laplacians in polarizable Carnot groups}\label{Subsect:carnot}
In this subsection, we first recall polar coordinates from \cite{BT} for polarizable Carnot groups (cf. \cite[Definition 2.12]{BT}). 
We then extend tools in geometric measure theory
in the context of homogeneous Carnot groups. After these preparations,  we use the new idea and \eqref{Equ:new-inequality} to 
derive \eqref{Equ:intr-dim-p} on the Hausdorff dimension of polar sets of potentials of nonnegative Radon measures for sub-Laplacians. 
Our main result on potentials is

\begin{theorem}\label{Thm:potential-carnot} Let $\G=(\R^N, \circ, \delta_\lambda)$ be a polarizable Carnot group.
Let $R^{\mu, D}_{\Delta_\G}$ be the potential of a nonnegative Radon measure $\mu$ supported in $D\subset \R^N$
for the sub-Laplacian $\Delta_\G$ in $\G$.  Suppose that  $S$ is a compact subset in $D$ and that
\begin{equation}\label{Equ:complete-carnot}
\lim_{t\to 0} R^{\mu, D}_{\Delta_\G}(\gamma(t)) = \infty \text{ and } \int_0^\delta (R^{\mu, D}_{\Delta_\G}(\gamma(t)))^\frac 2{Q-2} dt = \infty,
\end{equation}
for any curve $\gamma(t): [0, \delta]\to \R^N$ with $\gamma(0)\in S$ and $\gamma\in C[0, \delta]\cap C^1(0, \delta]$ 
that is horizontal with respect to the sub-Riemannian structure of $\G$ and
some $\delta >0$, where $Q$ is the homogeneous dimension of $\G$. Then 
$$
dim_{\mathcal{H}_{\Delta_\G}} (S) \leq \frac {Q-2}2.
$$
\end{theorem}

Here $d_{\Delta_\G}$ is the $\Delta_\G$-gauge given in \eqref{Equ:L-gauge} and $\mathcal{H}^s_{\Delta_\G}$ is the Hausdorff measure associated with $d_{\Delta_\G}$ in $\G$.


\subsubsection{Polarizable Carnot groups and polar coordinates} In this subsection we follow \cite{BT} to introduce polar
coordinates of horizontal radial curves on polarizable Carnot groups. 
To introduce the horizontal 
radial curves as flow lines, we let 
$$
\textup{S} = \{x\in \R^N: d_{\Delta_\G}(x) = 1\} \text{ and } \mathcal{Z} = \{0\}\bigcup \{x\in\R^N: \nabla_0 d_{\Delta_\G} (x) = 0\},
$$ 
where $\nabla_0 d_{\Delta_\G}$ is the horizontal gradient of the $\Delta_\G$-gauge $d_{\Delta_\G}$. In $\R^N\setminus\mathcal{Z}$, we consider the 
flow lines $\phi(s, g):  \R^N\setminus\mathcal{Z}\to  \R^N\setminus\mathcal{Z}$ that satisfies
\begin{equation} \label{Equ:flow-BT}
\left\{ \aligned \frac {\partial}{\partial s} \phi(s, g) & = \frac {d_{\Delta_\G}(\phi(s, g))}s \cdot 
\frac {\nabla_0 d_{\Delta_\G} (\phi(s, g))} {|\nabla_0 d_{\Delta_\G}(\phi(s, g))|}, \\
\phi(1, g)  & = g. \endaligned\right.
\end{equation}

\begin{lemma}\label{Lem:polar-coordinate} (\cite[Lemma 3.3]{BT}) Let $\G= (\R^N, \circ, \delta_\lambda)$ be a polarizable Carnot group. 
Then the flow lines $\phi(s, g)$ satisfies
\begin{itemize}
\item $d_{\Delta_\G}(\phi(s, g)) = s d_{\Delta_\G}(g)$ for $s>0$ and $g\in \R^N\setminus\mathcal{Z}$ (thus each curve $\phi(s, g)$ intersects with 
$\textup{S}$ in precisely one point);
\item the velocity  $\frac {\partial}{\partial s} \phi(s, g)$ is horizontal and the speed is independent of $s$;
\item $\det D \phi(s, g) = s^Q$ for $s>0$ and $g\in \R^N\setminus\mathcal{Z}$.
\end{itemize}
\end{lemma}

\begin{lemma}\label{Lem:polar-integral} (\cite[Theorem 3.7]{BT}) Let $\G = (\R^N, \circ, \delta_\lambda)$ be polarizable.
There exists a unique Radon measure $\sigma$ on $\textup{S}\setminus\mathcal{Z}$ such that for all $u\in L^1(\R^N)$
\begin{equation}\label{Equ:polar-integral}
\int_{\R^N} u(x) dx = \int_0^\infty \int_{\textup{S}\setminus\mathcal{Z}} u(\phi(s, v)) s^{Q-1} d\sigma(v) ds.
\end{equation}
\end{lemma}

It is important for us  that H-type groups (cf. \cite[Definition 5.1]{BT}) are polarizable by \cite[Proposition 5.6]{BT}, and that
H-type groups include the H-type groups $\mathsf{N}$ that are the nilpotent part of Iwasawa decomposition of rank one 
groups $\mathfrak{G}$ as in Theorem \ref{Thm:app-CR}.

\subsubsection{Some preliminaries in geometric measure theory in homogeneous Carnot groups} 
In this subsection we want to extend Lemma \ref{Lem:Falconer} and \ref{Lem:continuity-euclidean} in the previous subsection to be applicable to
the Hausdorff measures $\mathcal{H}^s_{\Delta_\G}$ 
associated with the $\Delta_\G$-gauge $d_{\Delta_\G}$ in homogeneous Carnot groups $\G$. First we want the
following extension of Lemma \ref{Lem:Falconer}.

\begin{lemma}\label{Lem:Falconer-carnot} Let $\G = (\R^N, \circ, \delta_\lambda)$ be a homogeneous Carnot group. 
Let $E$ be a closed subset of $\R^N$ with $\mathcal{H}^s_{\Delta_\G}(E)=\infty$. 
Then
\begin{itemize}
\item[(a)] If $c$ be a positive number, then there is a compact subset $F$ of $E$ such that $\mathcal{H}^s_{\Delta_\G}(F) = c$.
\item[(b)] There is a compact subset $F$ of $E$ such that $\mathcal{H}^s_{\Delta_\G}(F)>0$ and
\begin{equation}\label{Equ:h-regularity}
\mathcal{H}^s_{\Delta_\G}(F\cap B_{\Delta_\G}(x, r))\leq b r^s
\end{equation}
for $x\in \R^N, r\leq 1$ and some constant $b$, where $B_{\Delta_\G}(0, r)$ is the ball with respect to the $\Delta_\G$-gauge $d_{\Delta_\G}$.
\end{itemize}
\end{lemma}

\begin{proof}
First, (a) follows from \cite[Theorem 2]{LA}. 
This is because Carnot-Carath\'{e}redory spaces, which have the finite doubling property, are finite-dimensional according to
\cite{LA-d}. The proof of \cite[Theorem 2]{LA} uses the family of nets constructed in \cite[Theorem 1]{LA}.  In fact, using the family of nets
constructed in \cite[Theorem 1]{LA}, (b) can be proved as well,  
according to the proof of \cite[Theorem 5.4]{Fal}.
\end{proof}

Next let 
\begin{equation}\label{Equ:phi_L}
\phi_{\Delta_\G} (y) = \int_{F} \frac 1{d_{\Delta_\G}(y, p)^{t}}d\mathcal{H}^{s}_{\Delta_\G} (p). 
\end{equation}
We also want the following extension of Lemma \ref{Lem:continuity-euclidean} on the continuity of $\phi_\L$.

\begin{lemma}\label{Lem:continuity-carnot} Let $\G = (\R^N, \circ, \delta_\lambda)$ be a homogeneous Carnot group.
Suppose that  $F\subset\R^N$ be a compact subset which satisfies $\mathcal{H}^{s}_{\Delta_\G}(F) \in (0, \infty)$ 
and \eqref{Equ:h-regularity}. Assume that  $0\leq t< s$. Then $\phi_{\Delta_\G}(y)$ in \eqref{Equ:phi_L} is continuous.
\end{lemma}

\begin{proof} The proof goes with little changes from that of Lemma \ref{Lem:continuity-euclidean} based on Lemma \ref{Lem:Falconer-carnot}. 
Note that the $\Delta_\G$-gauge $d_{\Delta_\G}$ satisfies pseudo-triangle inequality (cf. \cite[Proposition 5.1.7]{BLU}).
\end{proof}


\subsubsection{Proof of Theorem \ref{Thm:potential-carnot}} After the preparations in the previous subsections, 
we are ready to use Theorem \ref{Thm:new-integral-gauge} to prove Theorem \ref{Thm:potential-carnot}.

\begin{proof}[Proof of Theorem \ref{Thm:potential-carnot}] The proof goes similar to that of Theorem \ref{Thm:new-SY-potential} in the previous
subsection. Assume otherwise that $dim_{\mathcal{H}_{\Delta_\G}}(S) > \frac {Q-2}2$ 
and consequently there is a compact subset $S'\subset S$ such that 
\begin{equation}\label{Equ:epsilon}
0 < \mathcal{H}_{\Delta_\G}^{\frac {Q-2}2 +\epsilon}(S') < \infty
\text{ and }
\mathcal{H}^{\frac {Q-2}2+\epsilon}_{\Delta_\G} (S'\cap B_{\Delta_\G}(x, r))\leq b_{\Delta_\G} r^{\frac {Q-2}2+\epsilon}
\end{equation}
for some $\epsilon >0$ by Lemma \ref{Lem:Falconer-carnot}. To contradict with \eqref{Equ:complete-carnot}, 
we want to show
$$
\int_0^\delta (R^{\mu, D}_{\Delta_\G} (x\circ \phi(s, v) )^\frac 2{Q-2}ds < \infty
$$
for some fixed and small $\delta>0$, some point $x\in S'$, and some $v\in \textup{S}\setminus\mathcal{Z}$, 
where $\phi(s, v)$ is the flow lines introduced in \eqref{Equ:flow-BT} and Lemma \ref{Lem:polar-coordinate}.  
For this, we use the H\"{o}lder's inequality
$$
\int_0^\delta (R^{\mu, D}_{\Delta_\G}(x\circ \phi(s, v))^\frac 2{Q-2} dr \leq (\int_0^\delta R^{\mu, D}_{\Delta_\G} (x\circ\phi(s, v)) s^\alpha ds)^\frac 2{Q-2}
 (\int_0^\delta s^{-\frac {2\alpha}{Q-4}}ds)^{1-\frac 2{Q-2}}
$$
and realize that it suffices to have $\int_0^\delta R^{\mu, D}_{\Delta_\G}(x\circ\phi(s, v)) s^\alpha ds < \infty$
when $\alpha < \frac {Q-4}2$. This step is not needed if $Q=4$. Next, for our purpose, it suffices 
that
\begin{equation}\label{Equ:Heisenberg-case}
\int_{S'}d\mathcal{H}_{\Delta_\G}^{\frac {Q-2}2+\epsilon}(x) \, \int_{\textup{S}\setminus\mathcal{Z}}d\sigma (v) \int_0^\delta 
R^{\mu, D}_{\Delta_\G}(x\circ \phi(s, v)) s^\alpha ds < \infty
\end{equation}
for some $\alpha < \frac {Q-4}2$. In the light of \eqref{Equ:polar-integral} in Lemma \ref{Lem:polar-integral}, to achieve  \eqref{Equ:Heisenberg-case}, we apply Fubini's Theorem and calculate
\begin{equation}
\aligned
&\int_{S'}d\mathcal{H}_{\Delta_\G}^{\frac {Q-2}2+\epsilon}(x) \, \int_{\textup{S}\setminus\mathcal{Z}}d\sigma (v) \int_0^\delta 
R^{\mu, D}_{\Delta_\G}(x\circ \phi(s, v)) s^\alpha ds \\
= & \int_{S'}d\mathcal{H}_{\Delta_\G}^{\frac {Q-2}2+\epsilon}(x) \, \int_{\textup{S}\setminus\mathcal{Z}}d\sigma (v) \int_0^\delta 
\int_D \frac {d\mu(y)}{d_{\Delta_\G} (y^{-1}\circ (x\circ\phi(s, v))^{Q-2}} s^\alpha ds  \\
 \leq &\int_D d\mu(y)\int_{S'} d\mathcal{H}_{\Delta_\G}^{\frac {Q-2}2+\epsilon}(x) \, \int_0^\delta ds 
 \int_{\textup{S}\setminus\mathcal{Z}} \frac 1{d_{\Delta_\G}(y^{-1}\circ(x\circ\phi(s, v))^{Q-2}} \, \frac {s^{Q-1}d\sigma (v)}{s^{Q-1-\alpha}}\\
\leq & C\int_D d\mu(y) \int_{S'}d\mathcal{H}_{\Delta_\G}^{\frac {Q-2}2 + \epsilon}(x) \int_{B_{\Delta_\G}(x, \delta)} 
\frac 1{d_{\Delta_\G}(y, u)^{Q-2}d_{\Delta_\G}(u, x)^{Q-1-\alpha}}du.
\endaligned
\end{equation}
Then, applying Theorem \ref{Thm:new-integral-gauge}, we get
$$
 \int_{B_{\Delta_\G}(x, \delta)} \frac 1{d_{\Delta_\G}(y, u)^{Q-2}
d_{\Delta_\G}(x, u)^{Q-1-\alpha}} du \leq \frac C{d_{\Delta_\G}(y, x)^{\frac {Q-2}2 + \frac \epsilon 2}}
$$ 
for $\alpha = \frac {Q-4}2 - \frac \epsilon 2$ and $\epsilon >0$ as in \eqref{Equ:epsilon}. 
Finally, in the light of Lemma \ref{Lem:continuity-carnot}, we know  
$$
\int_{S'}\frac 1{d_{\Delta_\G} (y, x)^{\frac {Q-2}2+\frac \epsilon 2}}d\mathcal{H}_{\Delta_\G}^{\frac {Q-2}2+\epsilon} (x) 
$$
is continuous with respect to $y$. This completes the proof.
\end{proof}


\section{Geometric applications}\label{Sect:geo-app}

In this section we develop our main applications in geometry based on Theorem \ref{Thm:potential-carnot}. Our proof of
Theorem \ref{Thm:potential-carnot} in the previous section is completely different from those that rely on PDE theory 
\cite{SY, JL, CQY-d, CQY, IMV, CH, Car20, GMS, SW, LX, LW}
as well as those that rely on the property of thin sets in potential theories
\cite{Hu, AH, MQ-g, MQ-2025, MQ-s}. In part, our new approach is devised to be the alternative to finding a horizontal curve that 
avoids a given thin set, when the property of thin sets observed in \cite[Theorem 1.1]{AH}, 
\cite[Theorem 8.1 in Chapter 2]{Mi}, and \cite[Theorem 1.2]{MQ-2025} in Euclidean spaces does not seem to be 
known in homogeneous Carnot groups or even in Heisenberg groups. 

\subsection{Yamabe problems in H-type groups $\mathsf{N}$}\label{Subsect:H-type}
Our applications of potentials of nonnegative Radon measures for sub-Laplacians to $\Delta_\G$-superharmonic functions that appear in Yamabe type 
problems in the geometries studied in  \cite{JL, Bq, W07, IMV, SW}. 


\subsubsection{Yamabe type problems in CR cases}\label{Subsubsect:CR-Yamabe}
One good place to start CR Yamabe problems is \cite{JL}. On a CR manifold, 
one may use a contact form $\theta$ to represent the distribution. Then the Levi form and the compatible complex structure $\mathbb{I} = \{I\}$ 
induces a Carnot-Carath\'{e}redory metric $g$ on the distribution, that is, 
$$
g(IX, Y) = d\theta (X, Y) \text{ and }  g(IX, IY) = g(X, Y) \text{ for horizontal vectors $X$ and $Y$.}
$$
This so-called pseudo-hermitian structure $(\theta, \mathbb{I})$ 
represents the underlined CR structure. The relation between a CR structure and its pseudo-hermitian representatives is analogous to that between
a conformal structure and its Riemannian representatives. For a pseudo-hermitian structure, there is a unique Tanaka-Webster connection. An alternative
approach based on Fefferman metrics was used in \cite{JL}, which seems to be in the same spirit as the Biquard connection in \cite{Bq}. The scalar curvature of such connection is called the Tanaka-Webster scalar curvature 
$R_\theta$. In fact, on a CR manifold, the Tanaka-Webster scalar curvature transforms under conformal 
changes as
\begin{equation}\label{Equ:scalar-CR}
(-b_n \Delta_b + R_\theta) u  = R_{\theta_u} u^\frac {Q+2}{Q-2}
\end{equation}
where $\Delta_b$ is the sub-Laplacian associated with the sub-Riemannian structure of $(\theta, g)$, 
$R_{\theta_u}$ is the Tanaka-Webster scalar curvature of the pseudo-hermitian structure associated with the conformal contact form
$\theta_u = u^\frac 4{Q-2}\theta$ for a positive function $u$,  and $Q=2n+2$ is the homogeneous dimension of the Heisenberg group $\bC^n\times \R$.  

On the Heisenberg group $\bC^n\times \R$, the pseudo-hermitian structure associated with the standard contact form $\theta_0$ is flat with respect to the 
Tanaka-Webster connection. Therefore we have
\begin{equation}\label{Equ:CR-yamabe}
-b_n\Delta_\H u = R_{\theta} u^\frac {Q+2}{Q-2} \text{ in $D\setminus S$},
\end{equation}
where $\theta = u^\frac 4{Q-2}\theta_0$, 
$u\in C^\infty(D\setminus S)$ and positive,  $D$ is some domain in $\bC^n\times\R$,  and $S$ is some compact subset in $D$.
Notice that such a conformal change introduces a conformal change of the corresponding pseudo-hermitian metric on the horizontal distribution. 

\subsubsection{Yamabe type problems in quaternionic CR cases}\label{Subsubsect:qc-Yamabe} 
In \cite{W07, IMV} after \cite{Bq}, similar to the extension made in \cite{JL}, the quaternionic CR Yamabe problems were introduced. 

A quaternionic contact structure on a manifold of dimensions $4n+3$ is given by a contact form $\theta = (\theta_1, \theta_2, \theta_3)$, which is 
a vector-valued 1-form, and three complex structures $\mathbb{I} = \{I_1, I_2, I_3\}$ 
on the horizontal distribution of $\theta$ such that there is the Carnot-Carath\'{e}redory metric $g$ on the horizontal distribution, where
$$
g(I_\alpha X, Y) = d\theta_\alpha (X, Y) \text{ and } g(I_\alpha X, I_\alpha Y) = g(X, Y)\text{ for each $\alpha = 1, 2, 3$}, 
$$ 
and 
$$
I_1\cdot I_2\cdot I_3 = - \text{Id} \text{ on the horizontal distribution.}
$$
A class of such conformal quaternionic contact structures represents a quaternionic CR structure. For each quaternionic 
contact structure, according to \cite[Theorem B]{Bq}, there is the unique connection $\nabla$, named as the Biquard connection. The scalar 
curvature $R_\theta$ of the Biquard connection transforms under conformal changes as
\begin{equation}\label{Equ:scalar-Q}
(- c_ n \Delta_b + R_\theta) u = R_{\theta_u} u^\frac {Q+2}{Q-2},
\end{equation}
where $\Delta_b$ is the sub-Laplacian of the sub-Riemannian structure of the Carnot-Carath\'{e}redory metric $g$, 
$R_{\theta_u}$ is the scalar curvature of the Biquard connection for the conformal contact form $\theta_u = u^\frac 4{Q-2} \theta$, and
$Q = 4n+6$ is the homogeneous dimension of the quaternionic Heisenberg group $\H^n\times\text{Im}\,\H$.

On the quaternionic Heisenberg group $\H^n\times \text{Im}\,\H$, the quaternionic contact structure associated with the standard contact 
form $\theta_0$ is flat with respect to the Biquard connection. Therefore we have
\begin{equation}\label{Equ:QC-yamabe}
- c_n\Delta_{\mathbb{Q}} u = R_{\theta} u^\frac {Q+2}{Q-2} \text{ in $D\setminus S$},
\end{equation}
where $\theta = u^\frac 4{Q-2}\theta_0$, 
$u\in C^\infty(D\setminus S)$ and positive,  $D$ is some domain in $\H^n\times \text{Im}\,\H$,  and $S$ is some compact subset in $D$.
Notice that such a conformal change introduces a conformal change of the Carnot-Carath\'{e}redory metric on the horizontal distribution. 


\subsubsection{Yamabe type problems in octonionic CR cases}\label{Subsubsect:OC-Yamabe} The story of octonionic CR cases goes 
very much like that of quaternionic CR cases, started in \cite{SW} recently. 

An octonionic contact structure on a manifold of dimensions $15$ is given by a contact form $\theta = (\theta_1, \theta_2, \cdots, \theta_7)$
and the complex structures $\mathbb{I} = \{I_1, I_2, \cdots, I_7\}$  on the horizontal distribution 
such that there is the Carnot-Carath\'{e}redory metric $g$ on the horizontal distribution, where
$$
g(I_\alpha X, Y) = d\theta_\alpha (X, Y) \text{ and }  g(I_\alpha X, I_\alpha Y) = g(X, Y) \text{ for each $\alpha = 1, 2, \cdots, 7$}, 
$$ 
and $\{I_1, I_2, \cdots, I_7\}$ satisfy the octonionic commutating relation \cite[(2.6)]{SW}.
A class of such conformal octonionic contact structures represents an octonionic CR structure. For each octonionic 
contact structure, according to \cite[Theorem B]{Bq}, there is the unique Biquard connection. The scalar 
curvature $R_\theta$ of the Biquard connection transforms under conformal changes as
\begin{equation}\label{Equ:scalar-O}
(-d_ n \Delta_b + R_\theta ) u = R_{\theta_u} u^\frac {Q+2}{Q-2}
\end{equation}
(cf. \cite[Corollary 1.1]{SW}), where $\Delta_b$ is the sub-Laplacian of the sub-Riemannian structure of the Carnot-Carath\'{e}redory metric $g$, 
$R_{\theta_u}$ is the scalar curvature of the Biquard connection for the conformal contact form $\theta_u = u^\frac 4{Q-2} \theta$, and
$Q= 22$ is the homogeneous dimension of the octonionic Heisenberg group $\mathbb{O}\times\text{Im}\,\mathbb{O}$.

On $\mathbb{O}\times\text{Im}\,\mathbb{O}$, the octonionic contact structure associated with the standard contact form $\theta_0$ 
is flat with respect to the Biquard connection. Therefore we have
\begin{equation}\label{Equ:OC-yamabe}
- d_n\Delta_{\mathbb{O}} u = R_{\theta} u^\frac {Q+2}{Q-2} \text{ in $D\setminus S$,}
\end{equation}
where $\theta = u^\frac 4{Q-2}\theta_0$, 
$u\in C^\infty(D\setminus S)$ and positive,  $D$ is some domain in $\mathbb{O}\times\text{Im}\,\mathbb{O}$,  and $S$ is some compact subset in $D$.
Notice that such a conformal change introduces a conformal change of the Carnot-Carath\'{e}redory metric on the horizontal distribution. 


\subsubsection{Yamabe constants and Yamabe functionals}\label{Subsubsect:yamabe-constant}
As mentioned in the above Section \ref{Subsubsect:CR-Yamabe}, the generalization of the Yamabe problems started with \cite{JL} in CR cases. Then the Yamabe problems have been also introduced in the quaternionic and octonionic CR  cases in \cite{W07, IMV, SW}. In each case, let
$$
\mathbb{L}_b = - b\Delta_b + R
$$
be the conformal sub-Laplacian for the contact structure representative $(\theta, g)$ in the class $([\theta], \mathbb{I})$, 
where $\Delta_b$ is the sub-Laplacian for the corresponding
sub-Riemannian structure, $R$ is the scalar curvature of the contact structure (with respect to Tanaka-Webster or the Biquard connection), and 
$b$ is a constant that depends on the structure group $\mathsf{N}$. 
Therefore, on a closed manifold
$M$ with the CR structure $([\theta], \mathbb{I})$ and a contact representative $(\theta, g)$, 
in the light of \eqref{Equ:scalar-CR}, \eqref{Equ:scalar-Q}, and \eqref{Equ:scalar-O}, we have
$$
\int_M u(- b\Delta_b + R)u dv_\theta = \int_M R_{\theta_u} dv_{\theta_u}
$$
where $dv_{\theta_u} = u^\frac {2Q}{Q-2} dv_\theta$ for $\theta_u = u^\frac 4{Q-2}\theta$. We note that 
the natural volume form associated with the contact form in each case is
\begin{equation}\label{Equ:volume}
dv_\theta = \left\{\aligned \theta\wedge (d\theta)^n & \text{ in CR cases}\\
\theta_1\wedge\theta_2\wedge\theta_3\wedge(d\theta_\alpha)^{2n} & \text{ for any $\alpha$ in quaternionic CR cases}\\
\theta_1\wedge\cdots\wedge\theta_7\wedge (d\theta_\alpha)^4 & \text{ for any $\alpha$ in octonionic CR cases,}\endaligned\right.
\end{equation}
where $(d\theta)^n, (d\theta_\alpha)^{2n},$ and $(d\theta_\alpha)^4$ are 
the volume form of the Carnot-Carath\'{e}redory metric on the horizontal distribution respectively. 
Thus, we consider the following global invariant of Yamabe type for the underlined CR structure
\begin{equation}\label{Equ:yamabe-type}
\lambda(M, [\theta], \mathbb{I}) = \inf_{\hat{\theta} \in [\theta]}\frac{\int_M R_{\hat{\theta}} dv_{\hat{\theta}}}{(\int_M dv_{\hat{\theta}})^\frac {Q-2}Q} 
= \inf_{u>0} \frac {\int_M u(-b\Delta_b + R_{\theta})u dv_\theta}{(\int_M u^\frac {2Q}{Q-2} dv_\theta)^\frac {Q-2}Q}.
\end{equation}
Please see \cite[(3.3)]{JL} in CR cases, \cite[(1.7)]{W07} and \cite{IMV} in quaternionic CR cases,  and \cite[(3.34)]{SW} in octonionic CR cases. 
\eqref{Equ:yamabe-type} not only defines a global invariant but also initiates the variational approach for the generalized Yamabe problems. 

It is easily seen that the sign of $\lambda(M, [\theta], \mathbb{I})$ matches with the sign of the first eigenvalue of the conformal sub-Laplacian
$-b\Delta_b+ R$ of $(\theta, g)$ in the class $([\theta], \mathbb{I})$ on compact manifolds. Moreover, motivated by the early work of
Aubin \cite{Au, LP} in the original Yamabe problems, using the first eigenfunction, it has been shown that

\begin{lemma}\label{Lem:aubin} (\cite[Theorem 6.2]{JL} \cite[Theorem 5.2]{W07} \cite[Proposition 4.2]{SW}) 
On a CR manifold $(M, [\theta], \mathbb{I})$ with respect to a H-type group $\mathsf{N}$ that is the nilpotent part of the Iwasawa decomposition 
of a rank one group $\mathfrak{G}$,
there is a contact structure $(M, \theta, \mathbb{I})$ in the class with nonnegative scalar curvature if $\lambda(M, [\theta], \mathbb{I})\geq 0$.
\end{lemma}

\subsection{$\Delta_\G$-superharmonic functions and geometric applications}\label{Subsect:main-app}

We are now ready to state our geometric applications of potentials. 

\begin{theorem}\label{Thm:app-CR} 
Assume $\mathsf{N}$ 
is a H-type group that is the nilpotent part of the Iwasawa decomposition of a rank one group $\mathfrak{G}$ among
$SU(n+1, 1)$, $Sp(n+1, 1)$ and $\mathbb{F}_{4(-20)}$. 
Let $S$ be a compact subset of a bounded and open subset $\Omega$ in $\mathsf{N}$. 
And let $\theta = u^\frac 4{Q-2}\theta_0$ be a 
conformal contact form on $\Omega\setminus S$, where $\theta_0$ is the standard contact form on $\mathsf{N}$ and $Q$ is the homogeneous 
dimension of $\mathsf{N}$. 
Suppose that the sub-Riemannian metric associated with $\theta$ is 
geodesically complete near $S$. Then
\begin{equation}
dim_{\mathcal{H}_{\Delta_{\mathsf{N}}}} (S) \leq \frac {Q-2}2
\end{equation}
when the scalar curvature $R_\theta$ of the conformal contact form $\theta$ is nonnegative.
\end{theorem}

To prove Theorem \ref{Thm:app-CR}, we follow \cite[Section 3.1-3.2]{MQ-2025} to derive the potential representations 
in order to use Theorem \ref{Thm:potential-carnot}. 
In fact we prove the stronger version that allows the scalar curvature $R_\theta$ 
to be negative somewhere as we did in \cite[Theorem 1.3]{MQ-2025}. Namely,

\begin{theorem} \label{Thm:app-CR-2.0} Assume $\mathsf{N}$ 
is a H-type group that is the nilpotent part of the Iwasawa decomposition of a rank one group $\mathfrak{G}$ among
$SU(n+1, 1)$, $Sp(n+1, 1)$ and $\mathbb{F}_{4(-20)}$. 
Let $S$ be a compact subset of a bounded and open subset $\Omega$ in $\mathsf{N}$. 
And let $\theta = u^\frac 4{Q-2}\theta_0$ be a 
conformal contact form on $\Omega\setminus S$, where $\theta_0$ is the standard contact form on $\mathsf{N}$ and $Q$ is the homogeneous 
dimension of $\mathsf{N}$. 
Suppose that the sub-Riemannian metric associated with $\theta$ is 
geodesically complete near $S$. Then
\begin{equation}
dim_{\mathcal{H}_{\Delta_{\mathsf{N}}}} (S) \leq \frac {Q-2}2
\end{equation}
provided that 
\begin{equation}\label{Equ:negativity-small}
R_\theta^- \in L^\frac {2Q}{Q+2} (\Omega\setminus S, dv_\theta)\bigcap L^p(\Omega\setminus S, dv_\theta)
\end{equation}
for some $p> \frac Q2$, where $R_\theta$ is the Webster scalar curvature of the pseudo-hermitian structure associated with the conformal contact 
form $\theta$ and $dv_\theta = u^\frac {2Q}{Q-2}dv_{\theta_0}$. Particularly, Theorem \ref{Thm:app-CR} holds.
\end{theorem}

As the first step, we want to understand preliminarily the singular behavior of the solution $u$ near $S$, 
which is the analog of 
\cite[Lemma 3.1]{MQ-g} and \cite[Lemma 1.3]{CY}, inspired by \cite[Proposition 8.1]{CHY}. 

\begin{lemma}\label{Lem:singularity-1} 
Under the assumptions of Theorem \ref{Thm:app-CR-2.0}, we have
$$
u(x) \to \infty \text{ as $x\to S$}.
$$
\end{lemma}

\begin{proof} The proof goes with little changes from \cite[Lemma 3.1]{MQ-g} and \cite[Lemma 1.3]{CY}.
To perform Moser iteration, one first needs the Sobolev-Stein inequality 
\cite[Theorem 5.9.2]{BLU} in homogeneous Carnot group $\mathsf{N}$ with the standard sub-Riemannian structure
\begin{equation}\label{Equ:Sobolev-Stein}
( \int |v|^\frac {2Q}{Q-2} dv_{\theta_0})^\frac {Q-2}{Q} \leq C \int |\nabla_\mathsf{N} v|^2 dv_{\theta_0}
\end{equation} 
for any $v\in C^\infty_c(\G)$. By the conformal invariant property, \eqref{Equ:Sobolev-Stein} implies
\begin{equation}\label{Equ:Sobolev-Stein-theta}
(\int |v|^\frac {2Q}{Q-2} dv_\theta)^\frac {Q-2}Q \leq C \int (|\nabla^\mathsf{N}_\theta v|^2 + \frac {R_\theta}{a_n} v^2)dv_\theta
\end{equation}
for $v\in C^\infty_c(D\setminus S)$, where $R_\theta$ is the scalar curvature with respect to the conformal contact structure $\theta$
and $a_n$ depends on the group $\mathsf{N}$ in the light of \eqref{Equ:scalar-CR}, \eqref{Equ:scalar-Q}, and \eqref{Equ:scalar-O}. The scalar curvature
equation with respect to the conformal contact structure $\theta$ is
\begin{equation}\label{Equ:reverse-Yamabe}
-a_n \Delta_\theta^\mathsf{N} u^{-1} + R_\theta u^{-1} = 0
\end{equation}
due to the fact that the scalar curvature $R_{\theta_0}$ of the standard contact structure 
on $\mathsf{N}$ vanishes. We rewrite \eqref{Equ:reverse-Yamabe} as
\begin{equation}\label{Equ:reverse-Yamabe-r}
-a_n\Delta_\theta^\mathsf{N} u^{-1} + R^+_\theta u^{-1} = R^-_\theta u^{-1}.
\end{equation}
Performing Moser iteration based on \eqref{Equ:Sobolev-Stein-theta} to \eqref{Equ:reverse-Yamabe-r} under the assumption 
\eqref{Equ:negativity-small} as in \cite[Lemma 3.1]{MQ-g} and \cite[Lemma 1.3]{CY} (cf. (8.5) in \cite[Proposition 8.1]{CHY}), we derive
$$
\|u^{-1}\|_{L^\infty (B_\theta (x, \frac 12))} \leq C (\int_{B_\theta(x, 1)} (u^{-1})^\frac {2Q}{Q-2}dv_\theta)^\frac {Q-2}{2Q} = C 
 (\int_{B_\theta(x, 1)} dv_{\theta_0})^\frac {Q-2}{2Q} \to 0
$$ 
as $x\to S$, provided that the sub-Riemannian metric associated with the conformal contact form $\theta$ is geodesically complete. This completes 
the proof.
\end{proof}

As the second step, we want to get the preliminary estimates of the Hausdorff dimensions of the singular set $S$ to at least 
confirm that $S$ is of Lebesgue 
measure zero. But, first we want to recall the Green Theorem on homogeneous Carnot groups $\G=(\R^N, \circ, \delta_\lambda)$ 
from \cite[(5.43c)]{BLU}, that is,
\begin{equation}\label{Equ:Green-thm}
\int_D (-\Delta_\G u) dx = \int_{\partial D} \nu \cdot A_\G\cdot\nabla^t u ds =  \int_{\partial D} \nu \cdot \nabla_\G uds.
\end{equation}
The following is the analog of \cite[Proposition 3.2]{MQ-2025} and \cite[Lemma 1.4]{CY}. 
 
 \begin{lemma}\label{Lem:singularity-2} Let $\G=(\R^N, \circ, \delta_\lambda)$ 
 be a homogeneous Carnot group. And let $D$ be a bounded domain in $\G$ and
 $S$ be a compact subset in $D$. Assume that $u$ satisfies
 \begin{equation}\label{Equ:semi-linear}
 -\Delta_\G u =  f (x)  \text{ in $D\setminus S$},
 \end{equation}
 where $f^- \in L^1(D)$. Suppose that $u(x)$ and $f(x)$ both are smooth away from $S$ in a neighborhood of $D$. And suppose that
 \begin{equation}\label{Equ:go-infty}
 u(x) \to \infty \text{ as $x\to S$}.
 \end{equation}
 Then
 \begin{equation}\label{Equ:p-hausdorff}
dim_{\mathcal{H}_{\Delta_\H}}(S) \leq Q-2.
 \end{equation}
 Particularly, under the assumptions of Theorem \ref{Thm:app-CR-2.0}, \eqref{Equ:p-hausdorff} holds.
 \end{lemma}
 
 \begin{proof} The proof follows that of \cite[Proposition 3.2]{MQ-2025} with little changes 
 (see also \cite[Lemma 1.4]{CY}). 
 As in \cite[Proposition 3.2]{MQ-2025}, the use of level sets based \eqref{Equ:go-infty} is essential. 
First, for $\alpha >0$ large and fixed,  $\eta\in C^\infty_c(\Sigma_\alpha)$ is a cutoff function such that $\eta \equiv 1$ on a neighborhood
of $S$, where $\Sigma_\alpha = \{x\in D: u(x) > \alpha\}$. Then, for any $\beta>0$, set
$$
u_{\alpha,\beta} = \left\{\aligned \beta & \text{ when $u\geq \alpha+\beta$}\\
u - \alpha & \text{ when $\alpha\leq u < \alpha+\beta$} \endaligned\right.
\text{ and } \phi_{\alpha, \beta} = u_{\alpha, \beta} - \beta + \beta (1 - \eta)
$$
(motivated from \cite[Lemma 1.2]{B-V}). Then $\phi_{\alpha, \beta}\in C^\infty_c(D\setminus S)$ and 
$$
\nabla_\G \phi_{\alpha, \beta} = \nabla_\G u_{\alpha, \beta} - \beta \nabla_\G \eta.
$$ 
Multiplying $\phi_{\alpha, \beta}$ to both sides of \eqref{Equ:semi-linear} and integrating on both sides, 
in the light of \eqref{Equ:Green-thm}, we have
$$
\int \nabla_\G u \cdot\nabla_\G \phi_{\alpha, \beta} dx = \int f (u_{\alpha, \beta} - \beta + \beta (1 - \eta))dx,
$$
which implies
$$
\int|\nabla_\G u_{\alpha, \beta}|^2 dx \leq \beta (\int \nabla_\G u \cdot\nabla_\G \eta dx + \int f(1 - \eta) dx + \int f^- dx)\leq C \beta
$$
for fixed $\eta$ and $\alpha$. Therefore
$$
\int |\nabla_\G \frac {u_{\alpha, \beta}}\beta|^2 dx \leq \frac C\beta \to 0 \text{ as $\beta\to\infty$}.
$$
This implies that the capacity $Cap_2(S) = 0$ and \eqref{Equ:p-hausdorff} follows from \cite[Corollary 4.6]{Cos}. For the introduction and discussion of 
$Cap_2$ we refer readers to \cite{BB} and \cite{Cos}. We remark here that,
due to the equivalence of homogeneous norms, the Hausdorff dimension is independent of the choice of homogeneous norms as the Hausdorff 
function.
\end{proof}
 
Finally, as the third step,  we want to verify that we may use potentials of nonnegative Radon measures for sub-Laplacians to represent the singular solutions to \eqref{Equ:semi-linear}, which is the analog of \cite[Lemma 3.3 and 3.4]{MQ-2025} (see also \cite[Lemma 1.5]{CY}).  
But, first, we recall from the proof of \cite[Lemma 3.3]{MQ-2025} the auxiliary function
\begin{equation}
\alpha_s (t) = \left\{\aligned t \quad\quad & \text{ when $t\leq s$}\\ \text{increasing} & \text{ when $t\in [s, 10s]$}\\ 2s 
\quad\quad & \text{ when $t\geq 10s$}
\endaligned\right.
\end{equation}
such that $(\alpha_s)'_t \in [0, 1]$ and $(\alpha_s)''_t \leq 0$ (first used in \cite{DHM}), and the calculation
\begin{equation}\label{Equ:sub-laplace-alpha}
- \Delta_\G \alpha_s(u) = - (\alpha_s)'' |\nabla_\G u|^2 + (\alpha_s)' (-\Delta_\G u).
\end{equation}

\begin{lemma}\label{Lem:singularity-3} Let $\G=(\R^N, \circ, \delta_\lambda)$ be a homogeneous Carnot group. 
And let $D$ be a bounded domain in $\G$ and
$S$ be a compact subset in $D$. Assume that $u$ satisfies
$$
-\Delta_\G u =  f (x)  \text{ in $D\setminus S$},
$$
where $f^- \in L^1(D)$. Suppose that $u(x)$ and $f(x)$ both are smooth away from $S$ in a neighborhood of $D$. And suppose that
$$
u(x) \to \infty \text{ as $x\to S$}.
$$
Then $f\in L^1(D)$, $-\Delta_\G u$ is a Radon measure on $D$, and $(-\Delta_\G u)|_S \geq 0$.
 \end{lemma}
 
\begin{proof} The proof follows the proofs of \cite[Lemma 3.3 and 3.4]{MQ-2025} (see also \cite[Lemma 1.5]{CY}).
The important point here is that the function $\alpha_s(u)$ is constant near the polar set $S$ since $u(x)\to\infty$ as $x$ approaches $S$. 
$s$ is going to be large, for example, $s > \max \{u(x): x\in \partial D\}$. After integrating both sides and using \eqref{Equ:Green-thm}, 
we have
$$
\int_{\partial D} \nu \cdot \nabla_\G u ds = \int_D (-\Delta_\G \alpha_s (u)) dx = 
\int_D (-(\alpha_s)'' |\nabla_\G u|^2 + (\alpha_s)' f) dx,
$$
which implies
\begin{equation}\label{Equ:integral-by-part}
\aligned
\int_D (-(\alpha_s)'' |\nabla_\G u|^2 & + (\alpha_s)' f^+) dx = \int_{\partial D} \nu\cdot\nabla_\G u ds + \int_D f^- (\alpha_s)' dx\\
&  \leq \int_{\partial D} \nu\cdot\nabla_\G u ds + \int_{D} f^-dx.
\endaligned
\end{equation}
Then, using Fatou's Lemma, from \eqref{Equ:integral-by-part}, we have
$$
\int_D f^+ dx = \int_D \lim_{s\to\infty} (\alpha_s)' f^+ dx \leq \int_{\partial D} \nu\cdot\nabla_\G u ds + \int_{D} f^-dx \leq C < \infty.
$$
Therefore $f\in L^1(D)$. In fact, from \eqref{Equ:integral-by-part}, we also have 
\begin{equation}\label{Equ:Delta-G}
\int_D |\Delta_\G \alpha_s(u)|dx  
\leq 2\int_{\partial D} \nu\cdot\nabla_\G u ds +3 \int_{D} f^-dx.
\end{equation}
By the definition, we want to consider, for
$\phi\in C^\infty_c(D)$,
\begin{equation}\label{Equ:passing-limit}
\aligned
\int & (-\Delta_\G u)\phi dx  = \int u (-\Delta_\G \phi) dx = \int \lim_{s\to\infty} \alpha_s(u) (- \Delta_\G\phi) dx \\
& = \lim_{s\to\infty}  \int \alpha_s(u) (-\Delta_\G\phi ) dx = \lim_{s\to\infty}\int (-\Delta_\G \alpha_s(u)) \phi dx,
\endaligned
\end{equation}
where one applies the dominant convergence theorem, using $u\in L^1(D)$ with the help from Lemma \ref{Lem:l-1} below. 
Hence, from \eqref{Equ:Delta-G} and \eqref{Equ:passing-limit}, we have
$$
|\int (-\Delta_\G u)\phi dx| \leq (2\int_{\partial D} \nu\cdot\nabla_\G u ds + 3\int_{D} f^-dx) \|\phi\|_{C^0},
$$
which implies $-\Delta_\G u$ is a Radon measure. Finally, when $\phi\geq 0$, from \eqref{Equ:passing-limit}, we have
$$
\aligned
\int (-\Delta_\G u)\phi dx & = \lim_{s\to\infty} \int (-\Delta_\G\alpha_s(u)) \, \phi dx
=  \lim_{s\to\infty}  \int (-(\alpha_s)'' |\nabla_\G u|^2 + (\alpha_s)' f)\, \phi dx\\
& \geq  \lim_{s\to\infty}  \int (\alpha_s)' f \, \phi dx\geq \int f\phi dx \to 0
\endaligned
$$
as $m(\text{supp}(\phi)\setminus S)\to 0$ with $\|\phi\|_{C^0} = 1$. This exactly means $(-\Delta_\G u)|_S \geq 0$.
\end{proof}

To complete the proof of Lemma \ref{Lem:singularity-3} we need 
\begin{lemma}\label{Lem:l-1} Under the assumption of Lemma \ref{Lem:singularity-3} we have $u\in L^1(D)$.
\end{lemma}

\begin{proof}  Let
$$
-\Delta_\G \alpha_s(u) = f_s.
$$
Then, from \eqref{Equ:Delta-G}, we know $\|f_s\|_{L^1(D)}$ is bounded uniformly with respect to $s$. 
Thanks to Theorem \ref{Thm:Riesz-l-superh} (\cite[Theorem 9.4.4]{BLU}), we may write
$$
\alpha_s(u) = \int_D f_s(y) \Gamma_{\Delta_\G} (y^{-1}\circ x) dy + h_s(x)
$$
for a function $h_s(x)$ that is $\Delta_\G$-harmonic in $D$, considering $f_s = f_s^+ - f_s^-$. First, using Fubini Theorem, we have
\begin{equation}\label{Equ:potential-l-1}
\int_D |\int_D f_s(y)\Gamma_{\Delta_\G}(y^{-1}\circ x) dy|dx \leq \int_Ddy |f_s(y)| \int_D dx \Gamma_{\Delta_\G}(y^{-1}\circ x) \leq C
\end{equation}
for some $C>0$ that is independent of $s$. Meanwhile, $\alpha_s(u) = u$ in $D\setminus D_{s_0}$, where 
$D_s=\{x\in D: u(x) > s\}\subset D$ when $s> s_0$ and $s_0$ is sufficiently large and fixed.
$$
\int_{D\setminus D_{s_0}} |h_s(x)|dx \leq \int_{D\setminus D_{s_0}} |u(x)| dx + C
$$
by \eqref{Equ:potential-l-1}, which implies $h_s$ is bounded in $D$ uniformly in $s$
in the light of the mean value theorem \cite[Theorem 5.6.1]{BLU} and 
the weak maximum principle \cite[Theorem 5.13.4]{BLU}.

Therefore, for a constant $C>0$, independeny of $s$, we got $\int_D |\alpha_s(u)|dx \leq C$.
Thus, by Fatou's Lemma, the proof of this lemma is complete.
\end{proof}

We remark here that the Radon measure $-\Delta_\G u$ in Lemma \ref{Lem:singularity-3} has the singular part
and therefore is not absolutely continuous with respect to Lebesgue measure on $\G$. Because 
$-\Delta_\G u = \mu$ in $D$ in distributional sense, while $-\Delta_\G u = f$ in $D\setminus S$, even though 
$f\in L^1(D)$. Now we are ready to present our proof of Theorem \ref{Thm:app-CR-2.0}.

\begin{proof}[Proof of Theorem \ref{Thm:app-CR-2.0}]
To summarize what we have from Lemma \ref{Lem:singularity-1} - \ref{Lem:singularity-3}, for $u$ in Theorem \ref{Thm:app-CR-2.0}, we may assume
 $$
 -\Delta_\mathsf{N} u = \mu \text{ in $D$}
 $$
 for some Radon measure $\mu$. Therefore 
 \begin{equation}\label{Equ:represented-potential}
 u (x) \leq  C\int_\Omega \frac {d\mu^+(y)}{d_{\Delta_\mathsf{N}}(y^{-1}\circ x)^{Q-2}} + h(x)
 \end{equation}
 where $\mu^+$ is nonnegative part of $\mu$ and $h(x)$ is $\Delta_\mathsf{N}$-harmonic in 
 $D$. Thus, Theorem \ref{Thm:app-CR-2.0} follows from Theorem \ref{Thm:potential-carnot}.
 \end{proof}
 
 \subsection{On quotients $\Omega(\Gamma)/\Gamma$ for convex cocompact subgroups $\Gamma$}\label{Subsect:subgroup}
 
 In this subsection we want to apply Theorem \ref{Thm:app-CR} to the quotient $\Omega(\Gamma)/\Gamma$ for 
 convex cocampact subgroups of real semisimple groups $\mathfrak{G}$ of rank one. 
 
 
 \subsubsection{Hyperbolic spaces and rank one groups}\label{Subsubsect:hyperbolic}
 Recall that rank one symmetric spaces of non-compact type, besides real hyperbolic spaces $\mathbb{H}^{n+1}$, are
 \begin{itemize}
 \item compex hyperbolic spaces $\H^{n+1}_\bC$;
 \item quaternionic hyperbolic spaces $\H^{n+1}_\mathbb{Q}$;
 \item the octonionic hyperbolic space $\H^2_\mathbb{O}$.
 \end{itemize}
Their isometry groups $\mathfrak{G}$ are the rank one groups $SU(n+1, 1)$, $Sp(n+1, 1)$, and $\mathbb{F}_{4(-20)}$ respectively. 
Let $\mathbb{F}$ be either the complex numbers $\bC$ or the quaternions $\H$ or the octonions $\mathbb{O}$ for convenience. Recall that
each hyperbolic space $\H_\mathbb{F}$ has a ball model and the boundary at infinity $\partial_\infty\H_\mathbb{F}$ is a sphere. 
When considering $\partial_\infty\H_\mathbb{F}$ as the Martin boundary of hyperbolic space $\H_\mathbb{F}$, 
$\Gamma$ naturally acts on $\partial_\infty\H_\mathbb{F}$ too, which is particularly simple and clear for rank one cases.
Then, for a discrete subgroup $\Gamma\subset\mathfrak{G}$ that acts discontinuously on $\H_\mathbb{F}$, 
for any $q\in \H_\mathbb{F}$, we consider the set of limit points of the orbit $\Gamma q$
$$
L(\Gamma) = \overline{\Gamma q}\cap \partial_\infty\H_\mathbb{F}.
$$
$L(\Gamma)$ turns out to be a closed invariant subset that is independent of $q$. Let
$$
\Omega(\Gamma) = \partial_\infty\H_\mathbb{F}\setminus L(\Gamma).
$$
Then, for a discontinuous discrete subgroup $\Gamma\subset \mathfrak{G}$, 
$\Omega(\Gamma)$ turns out to be the maximum open invariant subset in $\partial_\infty\H_\mathbb{F}$. $\Gamma$ is said to
be convex cocompact if the quotient $\Omega(\Gamma)/\Gamma$ is a compact manifold. Clearly $\Omega(\Gamma)/\Gamma$ inherits the 
standard CR structure on $\partial_\infty\H_\mathbb{F}$, which is usually referred to as being spherical or locally conformally flat. 

  
 \subsubsection{Cayley transforms}\label{Subsect:cayley}
To apply Theorem \ref{Thm:app-CR} in these cases, we need the Cayley transforms, which are the general stereographic projections 
(or the inverses) from the round spheres with a point removed to Euclidean spaces.
The full Cayley transform between two models of rank one groups $\mathfrak{G}$: 
Siegel domain model and unit ball model, is well defined, for example, in \cite[Section 3]{CDKR-98}. 
The contact structure changes conformally and therefore the CR structure remains the same under Cayley transform
$$
\C: \mathsf{N} =  \mathbb{F}^n\times \text{Im}\,\mathbb{F}\to \partial_\infty\H_\mathbb{F}\subset \mathbb{F}^n\times\mathbb{F}
$$
We refer readers to \cite[Appendix A]{FL}, \cite[(2.2)]{CLZ-Q}, and \cite[(2.1)]{CLZ-O} for the explicit Cayley transform $\C$ in each case.  
They are all invertible from the contact sphere with the south pole $S=(0, -1)$ removed. Assume, without loss of generality,
the south pole is in $\Omega(\Gamma)$. Therefore we have
$$
\C^{-1} (\Omega(\Gamma)\setminus\{(0, -1)\}) \subset \mathsf{N} \text{ and } \C^{-1}(L(\Gamma))\subset \mathsf{N}, 
$$
where $\C^{-1}(L(\Gamma))$ is compact in $\mathsf{N}$. 

\subsubsection{Conculsion}

In the light of Lemma \ref{Lem:aubin} and Theorem \ref{Thm:app-CR}, we have 
\begin{corollary}\label{Cor:subgroup} Let $\mathsf{N}$  be the H-type group that is the nilpotent part of the Iwasawa 
decomposition of a rank one group $\mathfrak{G}$ as in Theorem \ref{Thm:app-CR}. And let 
$\Gamma$ be a convex cocompact subgroup of $\mathfrak{G}$. 
Suppose that the Yamabe type constant $\lambda(\Omega(\Gamma)/\Gamma) \geq 0$ in \eqref{Equ:yamabe-type}. Then
\begin{equation}
dim_{\mathcal{H}}(L(\Gamma)) \leq \frac {Q-2}2,
\end{equation}
where $Q$ is the homogeneous dimension of the group $\mathsf{N}$.
\end{corollary}
\begin{proof}
 For the proof, we only need to point out that the conformal sub-Riemannian metric on $\C^{-1}(\Omega(\Gamma)\setminus\{(0, -1)\})$ 
 induced from that on
$\Omega(\Gamma)/\Gamma$, which is the one with nonnegative scalar curvature and conformal to the standard one, 
is geodesically complete when approaching to $\C^{-1}(L(\Gamma))$, because
$\Gamma$ makes infinitely many copies of compact $\Omega(\Gamma)/\Gamma$ in $\Omega(\Gamma)$. 
We also want to remark that the Hausdorff dimensions $dim_{\mathcal{H}}$ with respect to the control distance of $\mathsf{N}$ is the 
same as the Hausdorff dimensions $dim_{\mathcal{H}_{\Delta_\mathsf{N}}}$ with respect to the $\Delta_\mathsf{N}$-gauge on $\mathsf{N}$.
\end{proof}

Corollary \ref{Cor:subgroup} is the extension to all rank one groups $\mathfrak{G}$ from \cite[Proposition 2.4]{SY} in conformal groups. 
\cite{SY} is a seminal paper on the study of locally conformally flat manifolds through scalar curvature equations and
Yamabe problems.  
\cite{CCY} has studied the uniformization of spherical CR manifolds, using \eqref{Equ:scalar-CR}, in the spirit of \cite{SY}. See also works on Nayatani
metrics in \cite{Nay-c, Nay-h, WW, SW-n, SW}.


\noindent$\mbox{}^\dag$School of Mathematical Science and LPMC, Nankai University, Tianjin, China; \\e-mail: 
msgdyx8741@nankai.edu.cn 
\vspace{0.2cm}

\noindent $\mbox{}^\ddag$ Department of Mathematics, University of California, Santa Cruz, CA 95064; \\
e-mail: qing@ucsc.edu

\end{document}